\documentclass[12pt]{article}

\usepackage{amsfonts}

\def\be{\begin{equation}}
\def\ee{\end{equation}}
\def\l{\langle}
\def\r{\rangle}
\def\p{\parallel}
\def\R{I\!\!R}

\def\t{\tilde}
\def\nn{\{ 1, 2, \dots, n \}}
\def\mm{\{ 1, 2, \dots, m \}}
\def\ik{{i_k}}
\def\jk{{j_k}}
\def\bea{\begin{eqnarray}}
\def\eea{\end{eqnarray}}
\def\xls{x_{ls}}

\def\cR{{\cal R}}
\def\cN{{\cal N}}

\def\12{\frac{1}{2}}

\newtheorem{remark}{Remark}

\newtheorem{definition}{Definition}
\newtheorem{proposition}{Proposition}

\newtheorem{corollary}{Corollary}
\newtheorem{theorem}{Theorem}[section]

\newenvironment{proof}{\noindent\textbf{Proof.}
  }{\hspace*{\fill}$\spadesuit$ \\[2mm]}
\begin{document}

\begin{center}
{\Large \bf Convergence rates for Kaczmarz-type algorithms}

\vspace*{1cm}
{\small CONSTANTIN POPA}

\vspace*{1cm}
Ovidius University of Constanta, Blvd. Mamaia 124, Constanta 900527, Romania; {\tt cpopa@univ-ovidius.ro}

\vspace*{0.2cm}
              ``Gheorghe Mihoc - Caius Iacob'' Institute of Statistical Mathematics and Applied Mathematics of the Romanian Academy, Calea 13 Septembrie, Nr. 13, Bucharest 050711, Romania\\

\end{center}

{\bf Abstract.} 
In this paper we make a theoretical analysis of the convergence rates of Kaczmarz and Extended Kaczmarz projection algorithms for some of the most practically used control sequences. We first prove an at least  linear convergence rate  for the Kaczmarz-Tanabe and its Extended version methods (the one in which a complete set of projections using row/column index is performed in each iteration). Then we apply the main ideas of this analysis in establishing an at least sublinear, respectively linear convergence rate for the Kaczmarz algorithm with almost cyclic and the remotest set control strategies, and their extended versions, respectively. These results complete the existing ones related to the random selection procedures.



%---------------------------------------------------------------------------

%%%%%%%%%%%%%%%%%%%%%%%%%%%%%%%%%%%%%%%%%%%%%%%%%%%%%%%%%%%%%%%%%5
\section{Introduction}
\label{intro}

Kaczmarz projection algorithm is one of the most efficient iterative method for image reconstruction in computerized tomography. It has been proposed by the Polish mathematician Stefan Kaczmarz in his 3 pages short note \cite{kacz} (see also its English translation \cite{kaczengl}). For a square $n \times n$ nonsingular system of linear equations $Ax=b$ and $P_{H_i}$ the projection onto the hyperplane $H_i$ defined by its $i$-th equation (all these elements will be completely defined in the next section of the paper), the algorithm  originally proposed by Kaczmarz can be written as follows: given $x^{(0, 0)} \in \R^n$ compute
\be
\label{i-0}
x^{(0, s)} = P_{H_{s}}( x^{(0, s-1)}), s=1, 2, \dots, n, 
\ee
and set $ x^{(1, 0)} = x^{(0, n)}$. Then we replace in the above procedure $x^{(0, 0)}$ with $ x^{(1, 0)}$ and generate $ x^{(2, 0)}$ and so on. In this way, by successively projecting onto the hyperplanes $H_i$ the algorithm  generates a sequence of approximations  $(x^{(k, s)})_{k \geq 0, s=1, \dots, n} \in \R^n$ which  converges to the unique solution of the system $Ax=b$, independently on the choice of the initial approximation $x^{(0, 0)} \in \R^n$ and for any nonsingular matrix $A$. But unfortunately, for more than 10 years  remained unknown, and has been somehow reconsidered in few papers after 1948 (see \cite{cegcen},  \cite{sznader} and references therein). A crucial  moment in the evolution  of Kaczmarz's algorithm was the paper \cite{gbh} in which the algorithm has been rediscovered by the authors as the Algebraic Reconstruction Technique in computerized tomography. The next important  moment in considering Kaczmarz's method has been made by K. Tanabe in \cite{t71}. In his paper Tanabe  considers Kaczmarz algorithm with a complete projections set, visiting once each system hyperplane. More clear, starting from an approximation $x^k, k \geq 0$, the next one $x^{k+1}$ is generated as 
\be
\label{i-1}
x^{k+1}  = (P_{H_1} \circ \dots \circ P_{H_m})(x^k).
\ee
Tanabe proves that for any consistent system of equations $Ax=b, A: m \times n, b \in \R^m$, such that the rows of $A$ are nonzero, and any initial approximation $x^0 \in \R^n$ the sequence generated by (\ref{i-1}) converges to a solution of it, depending on $x^0$. We will call in the rest of the paper the algorithm (\ref{i-1}) as {\bf Kaczmarz-Tanabe} algorithm (KT, for short). Different than Kaczmarz-Tanabe, has been considered the single projection Kaczmarz method: start with $x^0 \in \R^n$, and for $k \geq 0$ select $i_k \in \{ 1, \dots, m \}$ and compute the next approximation $x^{k+1}$ as 
\be
\label{i-2}
x^{k+1}  =  P_{H_{i_k}}(x^k).
\ee
We will call in the rest of the paper the algorithm (\ref{i-2}) simply as {\bf Kaczmarz}  algorithm (for an almost  complete list of the selection procedures, together with a theoretical study see the papers  \cite{ycrow}, \cite{censta} (section 5.1), \cite{comb}, \cite{ccp11} and references therein). From these selection procedures we will consider in this paper the almost cyclic choice (with its particular case, cyclic choice) and call the corresponding algorithm  Almost Cyclic Kaczmarz, ACK for short. In the paper  \cite{ansorg} the author proposed a selection of $i_k$ such that the absolute value of the  $i_k$-th component of the residual is maximal with respect to the other components. We will call the corresponding algorithm Maximal Residual Kaczmarz, MRK for short.  Beside these selection procedures, a random choice of the projection index $i_k$ has been proposed in the paper \cite{strohm}, together with a theoretical analysis of the corresponding algorithm. We will call the corresponding algorithm Random Kaczmarz, RK for short. Unfortunately, all the above mentioned algorithms produce sequences convergent to solutions of the system $Ax=b$ only in the consistent case. And, although several considerations have been  made on the possibility of extending Kaczmarz-type algorithms to inconsistent systems $Ax=b$ (formulated in the least squares sense) (see \cite{censta}, \cite{cpbook} and references therein) an important contribution has been made by the author in \cite{cp95} (see also \cite{cp98}). Here has been porposed and theoretically analysed an extension of KT to inconsistent least squares problems which will be called in the present paper Extended Kaczmarz-Tanabe algorithm (EKT for short). Based on this extension, in the paper \cite{zou} the authors proposed and theoretically analysed an extension of the RK algorithm (called REK), whereas in the recent paper \cite{spcp} were proposed and theoretically analysed  similar extensions for the algorithms MRK and ACK (called MREK and ACEK, respectively). The scope of the present paper is to complete the analysis of the above mentioned algorithms from the view point of convergence rate. Until now, results are proved for the MRK, RK and REK confirming their linear convergence. In the present paper we prove that the KT, EKT, MRK and MREK algorithms have linear convergence rate, whereas the ACK and ACEK ones only sublinear convergence rate. In this way we get a complete image about the convergence properties of ones of the most used Kaczmarz-type algorithms. How to improve these properties or a similar analysis for other Kaczmarz-type algorithms (e.g. constrained versions) will be challenges for the near future reearch in the field. \\
According to the above considerations and aims, the paper is organized as follows: following the consideration from the well known monograph \cite{gau},  in section \ref{prelim} we present the basic definitions for the linear, superlinear and sublinear convergence rate of a sequence of vectors in $\R^n$, together with the other necessary definitions and notations used through the paper. Section \ref{consistent} is devoted to the analysis of the consistent case for the system $Ax=b$. We prove linear convergence rate for the KT algorithm, and sublinear one for the ACK method. In section \ref{inconsistent} we analyse the case of inconsistent least squares problems. We prove linear convergence rate for the EKT and MREK algorithms, and sublinear convergence rate for the ACEK method. 
%%%%%%%%%%%%%%%%%%%%%%%%%%%%%%%%%%%%%%%%%%%%%%%%%%%%%%%%%%%%%%%%%%%%%
%%%%%%%%%%%%%%%%%%%%%%%%%%%%%%%%%%%%%%%%%%%%%%%%%%%%%%%%%%%%%%%%%%%%%%
\section{Preliminaries}
\label{prelim}
We start te presentation of this section of the paper by introducing the concept of rate of convergence for convergent sequences of vectors in an Euclidean space $\R^q$. We used in this respect the well known monograph \cite{gau}. 
\begin{definition}  (\cite{gau}, Definition 4.2.1)
\label{def1} 
Let $(x^k)_{k \geq 0} \subset \R^n$ and $\xi \in \R^n$ such that $\lim_{k \rightarrow \infty} x^k = \xi$. One say that the sequence $(x^k)_{k \geq 0}$ converges to $\xi$ {\bf (at least) linearly} if 
\be
\label{008p}
\p x^k - \xi \p \leq \epsilon^k, \forall k \geq 0,
\ee
where $({\epsilon}_k)_{k \geq 0}$ is a sequence of positive real numbers satisfying
\be
\label{008}
\lim_{k \rightarrow \infty} \frac{\epsilon_{k+1}}{\epsilon_k}=\mu, ~~~0 < \mu < 1.
\ee
If (\ref{008p}) and (\ref{008}) hold with the inequality in (\ref{008p}) replaced by an equality, then $\mu$ is called the {\bf asymptotic error constant}. The phrase {\bf at least} relates to the fact that, in practice we have only inequality in (\ref{008p}), i.e. strictly speaking it is the sequence of bounds 
 $(\epsilon_k)_{k \geq 0}$ that converges linearly to $0$. 
\end{definition}
\begin{definition}  (\cite{gau}, Definition 4.2.2) 
\label{def2} 
 One say that the sequence $(x^k)_{k \geq 0}$ converges to $\xi$ with {\bf (at least) order $p \geq 1$} if (\ref{008p}) holds with
\be
\label{008pp}
\lim_{k \rightarrow \infty} \frac{\epsilon_{k+1}}{\epsilon^p_k}=\mu > 0.
\ee
(If $p=1$ one must assume, in addition, that $\mu < 1$). The constant $\mu$  is again refereed to as the {\bf asymptotic error constant} if we have equality in (\ref{008p}). If $\mu =1$ in (\ref{008}) the convergence is called {\bf sublinear}. If $\mu = 0$ in (\ref{008}) and (\ref{008pp}) does not hold for any $p > 1$ the convergence will be called {\bf superlinear}.
\end{definition}
\begin{remark}
\label{conjecture}
According to the above definitions, the almost sublinearity behavior appears when we have the limit in (\ref{008}) for $\mu=1$. However, it may happens (as it will be the case through the present paper) that this does not exactly hold, but the following situation occurs: let $\Delta_k = \frac{\epsilon_{k+1}}{\epsilon_k}, \forall k \geq 0$; it exists a subsequence $(\Delta_{k_s})_{s \geq 0}$ of $(\Delta_{k})_{k \geq 0}$ such that
\be
\label{008ppp}
\Delta_k = 1, \forall k \neq k_s ~~{\rm and}~~ \Delta_{k_s} = \delta \in [0, 1), \forall s \geq 0.
\ee
We would suggest to consider also this case a a sublinear behavior. Our argument is that, if also we would have  $\Delta_{k_s} = 1, \forall s \geq 0$, then we would satisfy the sublineariry assumptions. But, the fact that  $\Delta_{k_s} = \delta < 1, \forall s \geq 0$ tells us the at least on this subsequence the behavior is linear, thus better than sublinear.
\end{remark}
Let now  $A$ be an $m \times n$ matrix, $b \in \R^m$ a given matrix, and the consistent system 
\be
\label{1}
Ax = b.
\ee
In the rest of the paper $\l \cdot, \cdot \r, ~~\p \cdot \p$ will be the Euclidean scalar product and norm on some space $\R^q$, $A^T$ the transpose of $A$ with respect to $\l \cdot, \cdot \r$, and $\p A \p_2$  the spectral norm of  $A$ defined by 
$$
\p A \p_2 = \sup_{x \in \R^n \setminus \{ 0 \}} \frac{\p A x \p}{\p x \p}.
$$
  $A_i \neq 0, ~A^j \neq 0$ will be the $i$-th row, resectively $j$-th column of $A$ and we will suppose, without restricting the generality of the problem that
\be
\label{001}
A_i \neq 0, ~~A^j \neq 0.
\ee
We will denote the set of all solutions of (\ref{1})  by $S(A; b)$,  whereas $x_{LS}$ will be the minimal norm one. If $P_{C}$ is the orthogonal projection operator onto a convex closed set from an Euclidean space $\R^q$, and $\cN(A), ~\cR(A)$ are the null space and range of the matrix $A$ we know that the elements $x$ of $S(A; b)$ are of the form (see e.g. \cite{censta}, \cite{cpbook})
\be
\label{i-3}
x = P_{\cN(A)}(x) + x_{LS}, 
\ee
 and $x_{LS}$ is the unique solution which is orthogonal on $\cN(A)$, i.e. $\l x_{LS}, z \r = 0, \forall z \in \cN(A)$.
If the consistent right hand side $b$ of (\ref{1}) is perturbed with a noyse vector $r$ as 
\be
\label{0021}
\hat{b} = b + r, ~~b \in \cR(A), ~~r \in \cN(A^T), ~~{\rm i.e.}~~ r=P_{\cN(A^T)}(\hat{b}),
\ee
we reformulate (\ref{1}) as an  inconsistent least squares problem of the form
\be
\label{0022}
\p Ax - \hat{b} \p = \min_{z \in \R^n} \p Az - \hat{b} \p ~({\rm for}~ {\rm short},\ 
\p Ax - \hat{b} \p = \min!),
\ee
and we will denote by $LSS(A; b), ~x_{LS}$ its set of solutions and the minimal norm one. Similar properties as in (\ref{i-3}) characterize the elements of $LSS(A; b)$ and $x_{LS}$.
%%%%%%%%%%%%%%%%%%%%%%%%%%%%%%%%%%%%%%%%%%%%%%%%
%%%%%%%%%%%%%%%%%%%%%%%%%%%%%%%%%%%%%%%%%%%%%%%%
\section{The consistent case}
\label{consistent}
%%%%%%%%%%%%%%%%%%%%%%%%%%%%%%%%%%%%%%%%%%%%%%%%%%%%%%%%%%%%%%%%%%%%%%
\subsection{Kaczmarz-Tanabe algorithm}
\label{KT}

We will be concerned in this section with consistent systems of linear equations as (\ref{1}). $H_i =  \{ x \in \R^n, \l x, A_i \r = b_i \}$ will denote 
the hyperplane generated by the $i$-th equation of {1}, and $P_{H_i}, P_i$ the projections
\be
\label{a-20}
P_{H_i}(x) =  x  -   \frac{\l x, A_i \r - b_i}{\p A_i \p^2} A_i, 
~~~P_{i}(x) =  x  -   \frac{\l x, A_i \r}{\p A_i \p^2} A_i.
\ee
With these notations, the Kaczmarz algorithm considered by Tanabe in \cite{t71} can be written as follows.\\
{\bf Algorithm Kaczmarz-Tanabe (KT).} {\it Initialization.} Set $x^0 \in \R^n$.\\ 
{\it Iterative step.} For $k \geq 0$ do
\be
\label{tan2}
x^{k+1} = (P_{H_1} \circ \dots \circ P_{H_m})(x^k).
\ee
The following result is proved in \cite{t71}.
\begin{theorem}
\label{t2.1}
Let
 \be
\label{2.8}
Q_0 = I, \quad Q_i = P_1 P_2 \dots P_i,  \quad Q = P_1 \dots P_m,
\ee
\be
\label{2.9}
R = {\rm col}\left( {\frac {1}{\p A_1 \p^2}}Q_0 A_1, \dots, {\frac {1}{\p A_m \p^2}}Q_{m-1} A_m \right), 
 \t Q = Q \cdot P_{\cR(A^T)}
\ee
where $I$ is the unit matrix. Then
\be
\label{2.24p}
Q + RA = I, ~Q=P_{\cN(A)} + \t Q,  ~\p \t Q \p_2 ~<~ 1 ~{\rm and}
\ee
\be
\label{2.24pp}
x^{k+1} = Q x^k + Rb.
\ee
\end{theorem}
The result from the above theorem applies to a more general situation, as follows. Let $\Gamma \geq m$ be an integer, and $\gamma = \{ i_1, i_2, \dots, i_{\Gamma} \}$ a selection of projection indices such that
\be
\label{005}
\{ 1, 2, \dots, m \} \subset \gamma,
\ee
and the algorithm replaced by 
\be
\label{tan2p}
x^{k+1} = (P_{H_{i_{\Gamma}}} \circ \dots \circ P_{H_{i_1}})(x^k).
\ee
This algorithm corresponds to the ``extended'' system 
\be
\label{006}
A^{\gamma} x = b^{\gamma}, A^{\gamma}: \Gamma \times n, b^{\gamma} \in \R^{\Gamma}
\ee
with the elements given by
\be
\label{007}
A^{\gamma} = \left [ \begin{array}{c}
A_{i_1}\\
A_{i_2}\\
\dots\\
A_{i_{\Gamma}}\\
\end{array} \right ], ~~~
b^{\gamma} = \left [ \begin{array}{c}
b_{i_1}\\
b_{i_2}\\
\dots\\
b_{i_{\Gamma}}\\
\end{array} \right ] .
\ee
Because the system (\ref{1}) is consistent it results that so will be (\ref{007}). Moreover,  from (\ref{005}) we have that $S(A^{\gamma}; b^{\gamma}) = S(A; b)$ and, by directly applying the results from Theorem \ref{t2.1}, we construct the corresponding matrices $Q^{\gamma}, R^{\gamma}, \tilde{Q}^{\gamma}$  
 and get the results from (\ref{2.24p}) - (\ref{2.24pp}), in particular
\be
\label{2.26p}
\p \tilde{Q}^{\gamma} \p_2 < 1.
\ee
\begin{theorem}
\label{t1}
For the consistent system (\ref{1}),  $x^0 \in \R^n$, and $x^* \in S(A, b)$ such that 
\be
\label{002}
P_{\cN(A)}(x^*) =  P_{\cN(A)}(x^0)
\ee
it  holds 
\be
\label{2.40}
\p x^k - x^* \p ~\leq~ \p \t{Q}\p^k_2~ ~\p x^0 - x^* \p.
\ee
\end{theorem}
\begin{proof}
First of all we observe that
\be
\label{003}
P_{\cN(A)}(x^k)=P_{\cN(A)}(x^0) = P_{\cN(A)}(x^*), ~~x^* = Qx^* +Rb
\ee
see (\ref{003})
$$
\p x^k - x^* \p = \p Q x^{k-1} + Rb - x^* \p  = \p Q x^{k-1} - Q x^* \p = 
$$
$$
\p P_{\cN(A)}(x^{k-1}) + \tilde{Q}(x^{k-1}) -  P_{\cN(A)}(x^{*}) - \tilde{Q}(x^{*}) \p = \p \tilde{Q} (x^{k-1} - x^*) \p \leq
$$
\be
\label{004}
\p \tilde{Q} (x^{k-1} - x^*) \p ~\leq~ \p \tilde{Q} \p_2~ ~\p x^{k-1} - x^* \p.
\ee
\end{proof}
\begin{corollary}
\label{ct2}
The Kaczmarz - Tanabe algorithm (\ref{tan2}) has at least linear convergence.
\end{corollary}
\begin{proof}
If we define $\epsilon_{k} = \p \t{Q}\p^k_2 ~\p x^0 - x^* \p, ~\forall k \geq 0$, from (\ref{2.40}) we obtain 
\be
\label{2.40p}
\p x^k - x^* \p \leq \epsilon_{k}, ~\forall k \geq 0,
\ee
with 
$$\lim_{k \rightarrow \infty} \frac{\epsilon_{k+1}}{\epsilon_k} = \p \t{Q}\p_2 \in (0, 1),
$$
which completes the proof.
\end{proof}

%%%%%%%%%%%%%%%%%%%%%%%%%%%%%%%%%%%%%%%%%%%%%%%%%%%%%%%%%%%%%%%%
%%%%%%%%%%%%%%%%%%%%%%%%%%%%%%%%%%%%%%%%%%%%%%%%%%%%%%%%%%%%%%%%
\subsection{Kaczmarz single projection algorithm}
\label{KSP}
If we use in KT algorithm (\ref{tan2}) a single projection per iteration, following a projection index $i_k$ selected in an appropriate way we obtain the Kaczmarz algorithm with single projection, for short Kaczmarz (K).\\
{\bf Algorithm Kaczmarz}\\
{\it Initialization:} $x^0 \in \R^n$ \\
{\it Iterative step:} for $k = 0, 1, \dots$ select $i_k \in \{ 1, 2, \dots, m \}$ and compute $x^{k+1}$ as 
        \be
        \label{ksp}
        x^{k+1}= x^{k} -  \frac{\l x^{k}, A_{i_k} \r -  b_{i_k}}{\p A_{i_k} \p^2} A_{i_k}.
        \ee
				The most used selection procedures for the index $i_k$, that will be also analysed in the paper, are the following. 
\begin{itemize}
\item {\bf Cyclic  (\cite{censta}):} Set $i_{k} = k ~{\rm mod}~ m + 1$

\item {\bf Almost cyclic  (\cite{censta}):} Select $i_{k} \in \{ 1, 2, \dots, m \}$, such that it exists an integer $\Gamma$ with
\be \label{AC}
\{ 1, 2, \dots, m \} \subset \{i_{k+1}, \dots, i_{k+\Gamma}\} 
\ee
for every $k \geq 0$. It is clear that the cyclic selection procedure is a particular case of the almost cyclic one (for $\Gamma = m$).

\item {\bf Maximal Residual  (\cite{ansorg}):} Select $i_{k} \in \{ 1, 2, \dots, m \}$ such that
\be\label{MR}
|\l A_{i_k}, x^{k-1} \r - b_{i_k}| = \max_{1 \leq i \leq m} |\l A_{i}, x^{k-1} \r - b_{i}|.
\ee

\item {\bf Random (\cite{strohm}):} Let the set $\Delta_m \subset \R^m$ be defined by  
\be
\label{0010}
\Delta_m = \{ x \in \R^m, x \geq 0, \sum_{i=1}^m = 1 \},
\ee
define the discrete probability distribution 
\begin{equation}\label{RK}
p \in \Delta_{m},\ p_{i} = \frac{\|A_{i}\|^{2}}{\|A\|^{2}_{F}}, 
\ i = 1, \dots, m,
\end{equation}
and select $i_k \in \{ 1, 2, \dots, m \}$
\be\label{RK1}
i_{k} \sim p .
\ee 
\end{itemize}
According to the selection procedure used in Kaczmarz  algorithm, we will denote it by Cyclic Kaczmarz (CK), Almost Cyclic Kaczmarz (ACK), Maximal Residual Kaczmarz (MRK) and Random Kaczmarz (RK). The next result gives us information about the convergence rate of some of these algorithms.
 \begin{theorem}
\label{t3}
The following results are known.\\
(i) (\cite{ansorg}) Let $x^0 \in \R^n$, $x^* \in S(A; b)$ such that $P_{\cN(A)}(x^*) = $ $P_{\cN(A)}(x^0)$ and $(x^k)_{k \geq 0}$ the sequence generated with the MRK algorithm. Then it exists $0 <  \delta^2 < m$ independent on $k$ such that
\be
\label{0011}
\p x^{k+1} - x^* \p \leq \sqrt{ 1 - \frac{\delta^2}{m}} \p x^k - x^* \p, \forall k \geq 0.
\ee
(ii) (\cite{strohm}) Let $m \geq n$, $rank(A)=n$,  $x^0 \in \R^n$ and $(x^k)_{k \geq 0}$ the sequence generated by the algorithm RK. Then, it exists a constant $M \geq 1$ independent on $k$ such that 
\be
\label{0012}
\mathbb{E} \p x^k - x_{LS} \p^2 \leq \left ( 1 - \frac{1}{M^2} \right )^k 
\p x^0 - x_{LS} \p^2, \forall k \geq 0,
\ee
where $\mathbb{E}$ denotes the expectation.
\end{theorem}      
The previous result together with the definitions from section \ref{intro} give us the conclusion from the next result.
\begin{corollary}
\label{ct43}
Either  MRK or RK algorithm  has at least linear convergence, according to the euclidean norm, and expectation, respectively.
\end{corollary}
\begin{proof}
It results from (\ref{0011}) and (\ref{0012}), as in the proof of Corollary  \ref{ct2}.
\end{proof}
In the rest of this section we will prove a result related to the convergence rate of ACK (thus also CK) algorithm.
 \begin{theorem}
\label{t5}
Let $x^0 \in \R^n$ be an arbitrary initial approximation, $x^* \in S(A; b)$ such that $P_{\cN(A)}(x^*) = P_{\cN(A)}(x^0)$, and $(x^k)_{k \geq 0}$ the sequence generated with the algorithm ACK. Then, there exist $C \geq 0, ~\delta \in [0, 1)$ such that
\be
\label{3}
\p x^k - x^* \p ~\leq~ C ~\delta^{m_k},
\ee
where $m_k$ and $q_k \in \{0, 1, \dots, \Gamma -1 \}$ are (uniquelly) defined by 
\be
\label{2}
k = \Gamma \cdot ~m_k + q_k.
\ee
\end{theorem}
\begin{proof}
First of all we must observe that from the recurrence relation (\ref{ksp}) we obtain by mathematical induction that
\be
\label{0013}
P_{\cN(A)}(x^k) = P_{\cN(A)}(x^0) = P_{\cN(A)}(x^*), \forall k \geq 0.
\ee
Now, if $\Gamma$ is the almost cyclic constant from (\ref{ksp}) and $k \geq 0$ is arbitrary fixed we get
\be
\label{0014}
x^{k+\Gamma} = P_{H_{i_{k+\Gamma -1}}} \circ \dots \circ P_{H_{i_k}}(x^k).
\ee
We see that we are in the context from (\ref{005}) - (\ref{tan2p}), with 
$\gamma = \{ i_k, \dots, i_{k + \Gamma -1} \}$. Then we obtain the  system $A^{\gamma} x = b^{\gamma}$ and the matrices $Q^{\gamma}, R^{\gamma}, \tilde{Q}^{\gamma}$ with the properties (see (\ref{2.9}) - (\ref{2.24pp}))
\be
\label{0015}
Q^{\gamma} + R^{\gamma} A^{\gamma} = I, \tilde{Q}^{\gamma} = Q^{\gamma} P_{\cR((A^{\gamma})^T)}, Q^{\gamma} = P_{\cN(A^{\gamma})} + \tilde{Q}^{\gamma}, \p \tilde{Q}^{\gamma} \p_2 < 1,
\ee
\be
\label{0015p}
x^{k+\Gamma} = Q^{\gamma} x^k + R^{\gamma} b^{\gamma}.
\ee
Moreover, as the system $A^{\gamma} x = b^{\gamma}$ is also consistent,   $S(A^{\gamma}; b^{\gamma}) = S(A; b)$ and $x^* \in S(A; b)$, from the first equality in (\ref{0015}) we obtain
\be
\label{0016}
x^* = Q^{\gamma} x^* + R^{\gamma} A^{\gamma} x^* = Q^{\gamma} x^* + R^{\gamma} b^{\gamma}.
\ee
From (\ref{0014}) we then successively obtain, by also using (in this order) (\ref{0015p}), (\ref{0016})
$$
x^{k+\Gamma} = Q^{\gamma} x^k + R^{\gamma} b^{\gamma} = Q^{\gamma} x^k + x^* - Q^{\gamma} x^* = x^* + Q^{\gamma} (x^k - x^*). 
$$
Hence, because $P_{\cN(A)}(x^*) = P_{\cN(A)}(x^0)$ and $\cN(A^{\gamma}) = \cN(A)$ it results
$$
x^{k+\Gamma} - x^* = Q^{\gamma} (x^k - x^*) = \tilde{Q}^{\gamma} (x^k - x^*), 
$$
i.e. by taking norms
\be
\label{0017}
\p x^{k+\Gamma} - x^* \p \leq \p \tilde{Q}^{\gamma} \p_2 \p x^{k} - x^* \p, \forall k \geq 0.
\ee
But because it exists a finite number of subsets of the type $\gamma$, 
the inequality (\ref{3}) is then obtained by defining
\be
\label{0018}
\delta = \sup_{\gamma} \p \tilde{Q}^{\gamma} \p_2, ~C=\max_{1 \leq q \leq \Gamma -1} 
\p x^q - x^* \p.
\ee
and the proof is complete.
\end{proof}
\begin{corollary}
\label{ct5}
The ACK  algorithm has a sublinear convergence rate.
\end{corollary}
\begin{proof}
The inequality (\ref{3}) can be written
\be
\label{3p}
\p x^k - x^* \p ~\leq~ \epsilon_k, ~~\epsilon_k =  C ~\delta^{m_k}.
\ee
From (\ref{2}) it then results that\\
$m_0 = \dots = m_{\Gamma -1} = 0$\\
$m_{\Gamma} = \dots = m_{2 \Gamma -1} = 1$\\
$m_{2 \Gamma} = \dots = m_{3 \Gamma -1} = 2$\\
$ \dots  \dots  \dots  $\\
Thus
\be
\label{0019}
\frac{\epsilon_{k+1}}{\epsilon_k} = 1
\ee
excepting the subsequence $(\frac{\epsilon_{n \Gamma}}{\epsilon_{n \Gamma -1}})_{n \geq 1}$ for which we have 
\be
\label{0020}
 \frac{\epsilon_{n \Gamma}}{\epsilon_{n \Gamma -1}} = \delta \in [0, 1), \forall n \geq 1.
\ee
Then, the considerations  from Remark \ref{conjecture} apply and completes the proof.
\end{proof}
%%%%%%%%%%%%%%%%%%%%%%%%%%%%%%%%%%%%%%%%%%%%%%
%%%%%%%%%%%%%%%%%%%%%%%%%%%%%%%%%%%%%%%%%%%%%%
\section{The inconsistent case}
\label{inconsistent}

%%%%%%%%%%%%%%%%%%%%%%%%%%%%%%%%%%%%%%%%%%%%%%%%%%%%%%%%%%%%%%%%%%%%%%
\subsection{The Extended Kaczmarz-Tanabe algorithm}
\label{KTincon}

In this section we will consider the inconsistent least squares problem (\ref{0021}) - (\ref{0022}). The extension of KT algorithm (\ref{tan2}) to it was first proposed by the author in \cite{cp95}, and extensively studied in 
 \cite{cp98}. The main idea used for constructing the extension was to introduce a new step, in which a correction of the perturbed right hand side $\hat{b}$ is produced and then to correct it with this vector. This correction approximates $r$ from (\ref{0021}), the ``inconsistent'' component of $\hat{b}$ and is obtained by performing successive steps of KT algorithm for the consistent system $A^T y = 0$ (see \cite{cpbook} for details). \\
{\bf Algorithm Extended Kaczmarz-Tanabe (EKT)}. \\ {\it Initialization:} $x_0 \in \R^n, y^{0} = \hat{b}$;\\
{\it Iterative step:}
\bea
y^{k+1}&=&\Phi y^k,
\label{g31}\\
b^{k+1}&=&\hat{b} - y^{k+1},
\label{g32}\\
x^{k+1} &=& Q x^k + R b^{k+1}.
\label{g33}
\eea
with $Q, R$ from (\ref{2.8}) - (\ref{2.9}) and 
\be
\label{g34}
\Phi y = (\varphi_1 \circ \dots \circ \varphi_{n})(y), ~~\varphi_{j}(y) = y - \frac{\l y, A^j \r}{\p A^j \p^2} A^j, ~j = 1, \dots, n
\ee
 constructed as in (\ref{a-20})-(\ref{tan2}), but for the (consistent) system $A^T y =0$, as we already mentioned. Let $\t \Phi$ the application constructed as $\t Q$ in (\ref{2.9}) - (\ref{2.24p}) corresponding to $\Phi$ from (\ref{g34}), i.e. 
 \be
 \label{2.9fi}
 \t \Phi = \Phi P_{\cR(A)}, ~~\Phi = P_{\cN(A^T)} + \t \Phi, ~~\p \t \Phi \p_2 < 1.
 \ee
\begin{theorem}
\label{rate_ekt}
The Extended Kaczmarz - Tanabe algorithm (\ref{g31}) - (\ref{g33}) has at least linear convergence.
\end{theorem}
\begin{proof}
From \cite{cpbook}, Theorem 2.1, pages 124-125, if $e^k = x^k - (P_{\cN(A)} - x_{LS})$ is the error at the $k$-the iteration of the algorithm EKT, we know the relations
\be
\label{30}
e^k=\t Q e^{k-1} - R \t \Phi y^{k-1}, \forall k \geq 1, ~~~\t \Phi y^j = \t \Phi^2 y^{j-1}, \forall j \geq 2.
\ee
A recursive argument involving the first  equality in (\ref{30}), together with the relation $\t \Phi y^j = \t \Phi^{j+1} b$, which is obtained by using the second equality,  give us 
\be
\label{31}
e^k = \t Q^k e^0 - \sum_{j=0}^{k-1} \t Q^{k-j-1} R \t \Phi^{j+1} \hat{b}.
\ee
If we define $\delta = \max \{ \p \t Q \p_2, \p \t \Phi \p_2 \} < 1$ and take norms in (\ref{31}) we obtain
\be
\label{32}
\p e^k \p \leq \delta^k \p e^0 \p + \sum_{j=0}^{k-1} \delta^k \p R \p_2 \p \hat{b} \p = \delta^k (k+1) \p R \p_2 \p \hat{b} \p + \delta^k \p e^0 \p.
\ee
Now, if $\epsilon^k = \delta^k (k+1) \p R \p_2 \p \hat{b} \p + \delta^k \p e^0 \p$ we get 
$$\lim_{k \rightarrow \infty} \frac{\epsilon^{k+1}}{\epsilon^{k}} = \delta \in [0, 1),
$$
which shows us the at least linear convergence of the algorithm EKT and completes the proof.
\end{proof}
%%%%%%%%%%%%%%%%%%%%%%%%%%%%%%%%%%%%%%%%%%%%%%%%%%%%%%%%%%%%%%%%
%%%%%%%%%%%%%%%%%%%%%%%%%%%%%%%%%%%%%%%%%%%%%%%%%%%%%%%%%%%%%%%%
\subsection{Extended Kaczmarz single projection algorithm}
\label{EKSP}

 For extending  Kaczmarz  algorithm (\ref{ksp}) to the inconsistent leasr squares problem (\ref{0022}) we considered the same ideas from the previous subsection, but using only one projection in the $y$ and $x$ steps (\ref{g31}) and (\ref{g33}). We then obtained the following formulation of the method  (see for details \cite{spcp}).\\
{\bf Algorithm Extended Kaczmarz}\\
{\it Initialization:} $x^0 \in \R^n, y^0 = \hat{b}$ \\
{\it Iterative step:} Select the index $j_k \in \{ 1, \dots, n \}$ and set
       \be
       \label{EK-1}
       y^{k} = y^{k-1} - {\l y^{k-1}, A^{j_k} \r} A^{j_k}.
       \ee 
Update the right hand side as 
       \be
       \label{EK-2}
       b^k=\hat{b} - y^{k}.
       \ee
Select the index $i_k \in \{ 1, 2, \dots, m \}$ and compute $x^{k+1}$ as 
        \be
        \label{EK-3}
        x^{k}= x^{k-1} -  \frac{\l x^{k-1}, A_{i_k} \r -  b^k_{i_k}}{\p A_{i_k} \p^2} A_{i_k}.
        \ee
According to the selection procedures used in the above algorithm we distinguish the following three cases.
\begin{itemize}
\item {\bf Random Extended Kaczmarz (REK)} 
Define the discrete distributions 
\begin{equation}\label{0023}
p \in \Delta_{m},\ p_{i} = \frac{\|A_{i}\|^{2}}{\|A\|^{2}_{F}}, 
 i \in \mm,
 \ee
\begin{equation}\label{0023p}
q \in \Delta_{n},\ q_{j} = \frac{\|A^{j}\|^{2}}{\|A\|^{2}_{F}},
 j \in \nn,
\end{equation}
and sample in each step $k$ of the iteration (\ref{EK-1}), resp. (\ref{EK-3})
\be\label{jk_rand}
 j_{k} \sim q, ~~{\rm resp}~~ i_{k} \sim p.
\ee 
\item {\bf Maximal Residual Extended Kaczmarz (MREK)} 
 Select $\jk \in \nn$ and $\ik \in \mm$ such that
\be\label{jk_MR}
| \l A^{j_k}, y^{k-1} \r | = \max_{1 \leq j \leq n} | \l A^{j}, y^{k-1} \r |,
\ee
\be\label{ik_MR}
|\l A_{i_k}, x^{k-1} \r - b_{i_k}^{k}| = \max_{1 \leq i \leq m} 
|\l A_{i}, x^{k-1} \r - b_{i}^{k}|.
\ee
\item {\bf Almost cyclic Extended Kaczmarz (ACEK)} Select $j_k \in \nn$, $i_{k} \in \mm$, such that there exist  integers $\Gamma, \Delta$ with
\be \label{ACy}
\{ 1, 2, \dots, n \} \subset \{j_{k+1}, \dots, j_{k+\Delta}\}, 
\ee
\be \label{ACx}
\{ 1, 2, \dots, m \} \subset \{i_{k+1}, \dots, i_{k+\Gamma}\} 
\ee
for every $k \geq 0$. 
\end{itemize}
The following result was proved in \cite{zou} for the algorithm REK.
\begin{theorem}
\label{trek}
For any $A$, $\hat b$, and $x^0=0$, the sequence $(x^k)_{k \geq 0}$ generated by  REK Algorithm    converges  in expectation  to the minimal norm solution $\xls$ of (\ref{0022}) such that
\be
\label{raterek}
\mathbb{E} \big[\|x^k - \xls \|\big] ~\leq~ \bigg( 1 - \frac{1}{\hat{k}^2(A)} \bigg)^{\lfloor k/2\rfloor} (1 + 2 k^2(A)) \|x_{LS} \|^2,
\ee
where $\hat{k}(A) = \|A^+\|_2 \|A \|_F$ and $k(A) = \sigma_1 / \sigma_{\rho}$, where $\sigma_1 \geq \sigma_2 \geq \dots \geq \sigma_{r} > 0$ are the nonzero singular values of $A$ and ${\rho}=rank(A)$.
\end{theorem}
\begin{corollary}
\label{crek}
The algorithm REK   has at least linear convergence in  expectation.
\end{corollary}
\begin{proof}
It results from (\ref{raterek}), as in the proof of Corollary \ref{ct2}.
\end{proof}
\begin{theorem}
\label{t6}
The algorithm MREK has at least linear convergence.
\end{theorem}
\begin{proof}
Let $(x^k)_{k \geq 0}$ be the sequence generated with the MREK algorithm. 
According to the selection procedure (\ref{ik_MR}) of the projection index $i_k$  and (\ref{0021}) we successively obtain (see also section 1 of the paper \cite{ansorg})
$$
n |\l A_{i_k}, x^{k-1} \r - b_{i_k}^{k}|^2 \geq \sum_{1 \leq i \leq m} 
|\l A_{i}, x^{k-1} \r - b_{i}^{k}|^2 = \p A x^{k-1} - b^k \p^2 = 
$$
$$
\p A x^{k-1} - b \p^2 + \p r - y^k \p^2,
$$
hence
\be\label{40}
- |\l A_{i_k}, x^{k-1} \r - b_{i_k}^{k}|^2 \leq - \frac{1}{n}  \p A x^{k-1} - b \p^2 - \frac{1}{n}  \p r - y^k \p^2 .
\ee
In \cite{spcp}, Proposition 1 it is proved the equality
\be
\label{41}
\p x^k - x \p^2 = \p x^{k-1} - x \p^2 - \frac{\left(\l A_\ik, x^{k-1}\r-b_\ik\right)^2}{\|A_\ik\|^2}
+  \p \gamma_\ik \p^2, 
\ee
where
\be
\label{42}
\gamma_{i_k} = \frac{r_{i_k}-y^k_{i_k}}{\p A_{i_k} \p^2} A_{i_k}, ~{\rm and}~ x \in LSS(A; b) ~{\rm s.t.}~ P_{\cN(A)(x)} = P_{\cN(A)(x^0)}.
\ee
If $\delta$ is the smallest nonsingular value of $A$ and because $P_{\cN(A)(x^k)} = P_{\cN(A)(x^0)},$ $ \forall k \geq 0$ it holds that $x^k - x \in \cR(A^T)$, hence
\be
\label{43}
\p A x^{k-1} - b \p^2 \geq \delta^2 \p x^{k-1} - x \p^2.
\ee
Then, from (in this order) (\ref{41}), (\ref{40}),  the obvious inequality 
$$
\p \gamma_{i_k} \p^2 \leq \frac{\p r - y^k \p^2}{\p A_{i_k} \p^2}, 
$$
and (\ref{43}) we get
$$
\p x^k - x \p^2 \leq \p x^{k-1} - x \p^2 - \frac{1}{n} \frac{\p Ax^{k-1} - b \p^2}{\p A_{i_k} \p^2} - \frac{1}{n} \frac{\p r - y^k \p^2}{\p A_{i_k} \p^2} + \frac{\p r - y^k \p^2}{\p A_{i_k} \p^2} \leq
$$
\be
\label{44}
\left ( 1 - \frac{\delta^2}{n \cdot M} \right ) \p x^{k-1} - x \p^2 + 
\frac{1}{\mu} \left ( 1 - \frac{1}{n} \right) \p y^0 - r \p^2 \left ( 1 - \frac{\delta^2}{n} \right )^k,
\ee
where
\be
\label{45}
M = \max_{1 \leq i \leq m} \p A_i \p^2, ~~~\mu = \min_{1 \leq i \leq m} \p A_i \p^2.
\ee
If we introduce the notations
\be
\label{46}
\alpha = 1 - \frac{\delta^2}{n \cdot M} \in [0, 1), ~\beta = 1 - \frac{\delta^2}{n} \in [0, 1), ~C = \frac{1}{\mu} \left ( 1 - \frac{1}{n} \right) \p y^0 - r \p^2
\ee
from (\ref{44}) - (\ref{45}) we obtain
\be
\label{47}
\p x^k - x \p^2 \leq \alpha \p x^{k-1} - x \p^2 + \beta^k C, \forall k \geq 1.
\ee
From (\ref{47}), a recursive argument gives us 
$$
\p x^k - x \p^2 \leq \alpha^k \p x^0 - x \p^2 + \sum_{j=0}^{k-1} \alpha^j \beta^{k-j} C
$$
or, for $\nu = \max \{ \alpha, \beta \} \in [0, 1)$ 
\be
\label{48}
\p x^k - x \p^2 \leq \nu^k \left ( \p x^0 - x \p^2 + C k \right ), \forall k \geq 1.
\ee
If we define $\epsilon_k = \nu^k \left ( \p x^0 - x \p^2 + C k \right ), 
\forall k \geq 1$, we obtain that $\lim_{k \rightarrow \infty} \frac{\epsilon_{k+1}}{\epsilon_k} = \nu \in [0, 1)$, which gives us the at least linear convergence for MREK algorithm and completes the proof.
\end{proof}
\begin{theorem}
\label{t7}
The algorithm ACEK has at least sublinear convergence.
\end{theorem}
\begin{proof}
Let $(x^k)_{k \geq 0}$ be the sequence generated with the ACEK algorithm and $k \geq 0$ arbitrary fixed. From (\ref{EK-3}) it results
\be
\label{50}
x^{k+\Gamma} = P_{H_{i_{k+\Gamma - 1}}} \circ \cdots \circ P_{H_{i_{k}}}(b^k; x^k), ~{\rm where}~ P_{H_i(b^k; x)}= x -  \frac{\l x, A_{i} \r -  b^k_{i}}{\p A_{i} \p^2} A_{i}.
\ee
Let $\gamma = \{ i_k, \dots i_{k+\Gamma - 1} \}$ (see (\ref{005})), $A^{\gamma}, \hat{b}^{\gamma}$ defined with respect to (\ref{007}) and $Q^{\gamma}, R^{\gamma}$ the appropriate matrices from section \ref{KT} with the properties
$$
x^{k+\Gamma} = Q^{\gamma} x^k + R^{\gamma} \hat{b}^{\gamma}, ~~Q^{\gamma} = P_{{i_{k+\Gamma - 1}}} \circ \cdots \circ P_{{i_{k}}}
~~Q^{\gamma} + R^{\gamma} A^{\gamma} = I,
$$
with $P_i$ from (\ref{a-20}). If $x \in LSS(A; \hat{b})$ is such that 
\be
\label{52}
P_{\cN(A)}(x) = P_{\cN(A)}(x^0) = P_{\cN(A)}(x^k), ~\forall k \geq 0,
\ee
and because $\cN(A^{\gamma}) = \cN(A)$ we obtain
$$
P_{\cN(A^{\gamma})}(x^k - x) = P_{\cN(A)}(x^k-x) = 0.
$$
Let also $\t Q^{\gamma} = Q^{\gamma} P_{\cR((A^{\gamma})^T)}$ be defined according to (\ref{2.9}), with the properties (see (\ref{2.24p}))
\be
\label{53}
Q^{\gamma} = P_{\cN(A^{\gamma})} + \t Q^{\gamma} = P_{\cN(A)} + \t Q^{\gamma}, 
~~\p \t Q^{\gamma} \p_2 < 1.
\ee
From (\ref{52})-(\ref{53}) we get
\be
\label{54}
 Q^{\gamma}(x^k - x) = \t Q^{\gamma}(x^k - x).
\ee
Let now $x^*_k$ be the projection of $x^{k-1}$ on the ``consistent'' hyperplane of the problem (\ref{0022}), i.e. (see \cite{spcp}, eq (57))
\be
\label{55}
 x^k_* = x^{k-1} -  \frac{\l x^{k-1}, A_{i_k} \r -  b_{i_k}}{\p A_{i_k} \p^2} A_{i_k}.
\ee
According again to \cite{spcp}, proof of Proposition 1, we have the relations
$$
x^k  = x^k_* + \gamma_{i_k}, ~x^k_* - x = P_{{i_k}}(x^{k-1} - x), \forall k \geq 1,
$$
with $P_{i_k}$ from (\ref{a-20}). Hence
\be
\label{56}
x^k - x = x^k_* - x + \gamma_{i_k} = P_{{i_k}}(x^{k-1} - x) + \gamma_{i_k}, ~\forall k \geq 1.
\ee
from (\ref{56}) we obtain $\forall k \geq 1, j=0, 1, \dots, \Gamma -1$
\be
\label{57}
x^{k+\Gamma - j} - x = P_{{i_{k+\Gamma - j-1}}}(x^{k+\Gamma - j -1} - x) + \gamma_{{i_{k+\Gamma - j-1}}}.
\ee
A recursive argument gives us, by also using (\ref{a-20}) and (\ref{54})
$$
x^{k+\Gamma} = P_{k+\Gamma -1} \circ \cdots \circ P_{i_k}(x^k - x) + 
\sum_{j=1}^{\Gamma} \Pi_j \gamma_{i{k+\Gamma - j}} =
$$
\be
\label{57p}
Q^{\gamma}(x^k-x) + \sum_{j=1}^{\Gamma} \Pi_j \gamma_{i{k+\Gamma - j}} = \t Q^{\gamma}(x^k-x) + \sum_{j=1}^{\Gamma} \Pi_j \gamma_{i{k+\Gamma - j}}, \forall k \geq 1,
\ee
where
\be
\label{58}
\Pi_1=I, \Pi_2 = P_{i_{k+\Gamma -1}}, \dots, \Pi_{\Gamma}= P_{i_{k+\Gamma -1}} P_{i_{k+\Gamma -2}} \dots P_{i_{k+1}}.
\ee
The applications $\Pi_j$ are products of orthogonal projections (and $\Pi_1=I$), thus
\be
\label{59}
\p \Pi_j \p_2 \leq 1, \forall j = 1, \dots, \Gamma,
\ee
whereas for $\gamma_{i_l}$, from \cite{spcp}, eq. (41) it exists $\gamma \in [0, 1)$ such that
\be
\label{60}
\p \gamma_{i_l} \p \leq M \gamma^{q_l}, ~{\rm if}~ l=\Gamma \cdot q_l + s_l, ~s_l \in \{ 0, \dots, \Gamma -1 \}.
\ee
Then, by taking norms in (\ref{57p}) and using (\ref{58})-(\ref{60}) we obtain
\be
\label{61}
\p x^{k+\Gamma} - x\p \leq \p \t Q^{\gamma} \p_2 \p x^k - x \p + \sum_{j=1}^{\Gamma} M \gamma^{q_{k+\Gamma - j}},
\ee
from which we get with $\delta$ from (\ref{0018})
\be
\label{62}
\p x^{q \Gamma} - x \p \leq \delta \p x^{(q-1) \Gamma} - x \p + M \Gamma \gamma^{q-1}, \forall q \geq 1.
\ee
A recursive argument gives us from (\ref{62}) and $\mu = \max \{ \delta, \gamma \} \in [0, 1)$
\be
\label{63}
\p x^{q \Gamma} - x \p \leq \mu^q \p x^{0} - x \p + M \Gamma^2 \mu^{q-1}, \forall q \geq 1.
\ee
by using the same procedure as before we also get
\be
\label{64}
\p x^{q \Gamma + j} - x \p \leq \mu^q \p x^j - x \p + \gamma^2 M \mu^{q-1}, \forall q \geq 1, j = 0, 1, \dots, \Gamma - 1.
\ee
From (\ref{63}) and (\ref{64}) it then results for any $k \geq 1$
\be
\label{65}
\p x^k - x \p \leq \mu^{q_k} \alpha + \Gamma^2 M \mu^{q_k-1},
\ee
where
\be
\label{66}
k=q_k \cdot \gamma + r_k, ~r_k \in \{ 0, \dots, \Gamma -1 \}, ~\alpha = \max_{0 \leq j \leq \Gamma - 1} \p x^j - x \p.
\ee
now if we define 
\be
\label{67}
\epsilon_k = \mu^{q_k} \alpha + \Gamma^2 M \mu^{q_k-1},
\ee
from (\ref{65})-(\ref{66}) it results that
\be
\label{68}
\frac{\epsilon_{j+1}}{\epsilon_j} = 1, ~\forall j \neq q \gamma - 1, q \geq 1 ~{\rm and}~ \frac{\epsilon_{q \Gamma}}{\epsilon_{q \Gamma - 1}} = \mu \in [0, 1)
\ee
which gives us the at least sublinear convergence of the algorithm ACEK and completes the proof.
\end{proof}
%%%%%%%%%%%%%%%%%%%%%%%%%%%%%%%%%%%%%%%%%%%%%%%%%%%%%%%%%%%%%%%%%%%%%%%%%%%%%%
%%%%%%%%%%%%%%%%%%%%%%%%%%%%%%%%%%%%%%%%%%%%%%%%%%%%%%%%%%%%%%%%%%%%%%%%%%%%%%
\section{Final comments}
\label{final}

{\bf 5.1} In this paper we tried to fill-in the existing gap related to the convergence rates analysis of Kaczmarz-type algorithms. We first analysed the Kaczmarz - Tanabe (KT) algorithm, in which a complete set of projections using  each row index once is performed in each iteration, and we obtained for it an at least linear convergence rate. This result allowed us to analyse the Kaczmarz method with almost cyclic selection of indices (ACK), for which we obtained an at least  sublinear   one. For the ramdom choice (RK algorithm) and the remotest set control one (MRK algorithm) there were already obtained results, saying that both have at least linear convergence rate (for the first one in expectation; see \cite{strohm} and \cite{ansorg}, respectively). \\
The second part of the paper was devoted to the analysis of Extended Kaczmarz type algorithms. Our first result was given for the EKT algorithm, the extension of the KT one, proposed by the author in \cite{cp95} (see also \cite{cp98}). We obtained for it an at least linear convergence rate. This result together with the considerations from \cite{spcp} allowed us to prove at least linear, resp. sublinear convergence rate for the extended versions of MRK and ACK algorithms, respectively. For the extended version of the RK algorithm an at least linear convergence rate in expectation was already shown in \cite{zou}.

{\bf 5.2} Although the selection procedures considered in the paper are among the most used ones in practical applications, there are much more possibilities in this respect. For an almost complete overview see \cite{ycrow}, \cite{censta} (section 5.1), \cite{comb}, \cite{ccp11}. An important problem when considering a selection procedure seems to be the following: {\it ``sooner or latter'' during the iterations each (row) projection index $i_k$ must appear}. This was clearly formulated in \cite{ccp11} as follows.
\begin{definition}
Let ${I\!\!N}$ be the set of natural numbers. 
Given a monotonically increasing sequence $\{\tau_{k}\}_{k=0}^{\infty} \subset {I\!\!N}$, a mapping $i:{I\!\!N} \rightarrow \{1, 2, \dots, m\}$ is called a {\em control with respect to the sequence $\{\tau_{k}\}_{k=0}^{\infty}$} if it defines a {\em control sequence $\{i(t)\}_{t=0}^{\infty}$}, such that for all $k \geq 0$,
\begin{equation}\label{def:inclusion}
\{1, 2, \dots, m\} \subseteq \{i(\tau_k), i(\tau_k + 1), \dots, i(\tau_{k + 1} - 1)\}.
\end{equation}
The set $\{ \tau_k, \tau_k + 1, \dots, \tau_{k+1} - 1 \}$ is called the $k$-th {\em window} (with respect to the given sequence $\{\tau_{k}\}_{k=0}^{\infty}$) and  $C_k = \tau_{k + 1} - \tau_k$ its {\em length}. A control for which $\{ C_k \}_{k \geq 0}$ is an unbounded sequence is called an {\bf expanding} control.
\end{definition}
It is clear that the Almost cyclic control (with the cyclic one as particular case) fits directly into the above definition and is a bounded control. At least on author's knowledge it is not yet know if the random and maximal residual (remotest set) controls fit into the above definition and are bounded or expanding controls. In this respect we first point out a remark from \cite{chg2009}, related to the RK algorithm and saying that if ``{\it  the norm associated with one equation is very much larger than the
norms associated with the other equations, then the progress made by Algorithm  RK 
towards a solution would be poor due to the fact that the random selection would
keep selecting the same equation most of the time}'', and probably other rows with much  smaller norms will never be selected. A second remark is made in \cite{ioana}, in which the authors prove the following result. 
\begin{proposition}
\label{ipcp}
Let $A$ be $m \times n$, with $m \leq n$ and $rank(A)=m$. Let $x^0 \in \R^n$ an initial approximation and $x_{LS}$ the minimal norm solution of the system (\ref{1}) be expressed in the basis $\{ A_1, \dots, A_m \}$ of $\cR(A^T)$ as 
\be
\label{1000}
x_{LS} - P_{{\mathcal{R}}(A^T)}(x^0) = \sum_{i = 1}^m \gamma_i A_i, \gamma_i \in \R.
\ee
If
\begin{equation}\label{case1:sum}
 \gamma_i \neq 0, \forall  i \in \{1, 2, \dots, m\},
\end{equation}
then, for any $i\in \{1, 2, \dots, m\}$, it exists $k \geq 0$ (depending on $i$), such that in the algorithm MRK we have $i_k = i$.
\end{proposition}
The above result tells us that, in the hypothesis (\ref{case1:sum}) the remotest set control is a kind of {\it expanding} control. Moreover, the assumption (\ref{case1:sum}) can be fulfilled if we define $x^0 = \sum_{i=1}^m \alpha_i A_i$, with $\alpha_i \neq 0, \forall i$ ``enough big'' in absolute value.

{\bf 5.3} One interesting challenge for the  near future works in this direction would be to extend the above analysis related to RK and MRK algorithms, and also to the other types of controls from the above cited papers.

{\bf 5.4} The results from our paper are only theoretical. We are not making numerical experiments and comparisons. It does not exists an universal efficient algorithm, that overpasses all the other methods, for any system. In spite of the theoretical convergence rate results, an efficient implementation together with an appropriate class of problems serves for a specific algorithm and can make it better than the others in that specific context.


\begin{thebibliography}{99}

\bibitem{ansorg}
{\rm  Ansorge R.}, 
{\em Connections between the Cimmino-method and the Kaczmarz-method for the solution of singular and regular systems of equations}, Computing, {\bf 33} (1984),  367-375.

\bibitem{cegcen}
{\rm Cegielski A.,  Censor Y.},  {\em Projection methods: An annotated bibliography of books and reviews}, Optimization 64: 2343-2358, (2015).

\bibitem{ycrow}
{\rm Censor, Y.}, {\em Row-action methods for huge and sparse systems and their applications}, SIAM Review, 23 (1981), 444-466.

\bibitem{censta}
{\rm Censor, Y., Zenios, S.},  {\em Parallel Optimization: Theory, Algorithms and Applications},  Oxford Univ. Press (1997)

\bibitem{chg2009}
{\rm Censor Y., Herman G. T., Jiang M.}, {\it A note on the behavior of the randomized Kaczmarz algorithm of Strohmer and Vershynin},  J. Fourier Anal. Appl. 15(2009), 431 - 436.

\bibitem{ccp11}
{\rm Censor, Y.  Chen, W.,  Pajoohesh, H.}, {\em Finite convergence of a subgradient projections method with expanding controls},  Appl. Math. Optim.,
  {\bf 64} (2011), 273--285.

\bibitem{comb}
{\rm Combettes P.}, {\em Hilbertian convex feasibility problem: Convergence of projection methods},  Appl. Math. Optim.,
  {\bf 35} (1997), 311--330.

\bibitem{gau}
{\rm Gautschi W.}, {\em Numerical analysis. Second edition},  Birkh\" auser, Boston, 2012.

\bibitem{gbh}
{\rm Gordon R.,   Bender R., Herman G.T.}, {\em Algebraic Reconstruction Techniques (ART) for three-dimensional electron microscopy and X-ray photography}, Journal of Theoretical Biology 29, 1970, 471-481.

\bibitem{kacz}
{\rm Kaczmarz S.}, {\em Angen\" aherte Aufl\" osung von Systemen linearer Gleichingen}, Bull. Intern. Acad. Polonaise Sci. Lett., Cl. Sci. Math. Nat. A, 35, 1937, 355-357.

\bibitem{kaczengl}
{\rm Kaczmarz S.}, {\em Approximate solution of systems of linear equations}, International Journal of Control 57, 1993, 1269-1271.

\bibitem{spcp}
{\rm Petra S.,  Popa C.},  {\em Single projection Kaczmarz extended algorithms }, Numerical Algorithms, 73(3)(2016), 791-806; 

\bibitem{ioana}
{\rm  Pomparau I., Popa C.}, {\em A note on remotest set control selection procedure in Kaczmarz iteration}, unpublished manuscript (2016).

\bibitem{cp95}
 {\rm Popa C.}, {\em  Least-Squares Solution of Overdetermined Inconsistent Linear Systems using Kaczmarz's Relaxation},  Int. J. Comp. Math. 55(1-2), 79-89 (1995)

\bibitem{cp98}
{\rm Popa C.},  {\em Extensions of block-projections methods with relaxation parameters to inconsistent and rank-defficient least-squares problems}, {\it B I T} Numer. Math.,  38(1), 151-176 (1998)



\bibitem{cpbook}
{\rm Popa C.} - {\em Projection algorithms - classical results and developments. Applications to image reconstruction}, Lambert  Academic Publishing - AV Akademikerverlag GmbH \& Co. KG, Saarbr\" ucken, Germany, 2012

%\bibitem{suli}
%{\rm S\" uli E.,  Mayers D.},  {\em An introduction to numerical analysis},  Cambridge University Press, New York, 2003.

\bibitem{strohm}
{\rm Strohmer, T., Vershynin, R.},  {\it A randomized Kaczmarz algorithm with exponential convergence}, J. Fourier Anal. Appl. 15(2009), 262 - 278.

\bibitem{sv}
{\rm Strohmer, T., Vershynin, R.},  {\it Comments on the randomized Kaczmarz
method}, unpublished manuscript, 2009.

\bibitem{sznader}
{\rm Sznader R.}, {\it Kaczmarz algorithm revisited}, DOI: 10.4467/2353737XCT.15.220.4425. 

 \bibitem{t71}
{\rm Tanabe K.},
{\em  Projection Method for Solving a Singular System
of Linear Equations and its Applications}, Numer. Math.,
 {\bf 17} (1971), 203-214.

\bibitem{zou}
Zouzias, A., Freris, N., {\it Randomized Extended Kaczmarz for Solving Least Squares}, 
SIAM J. Matrix Anal. Appl. 34(2), 773 - 793 (2013)


\end{thebibliography}
\end{document}